\documentclass[12pt]{article}
\usepackage{amsmath, amsthm}
\usepackage[top=4cm,left=3cm,right=3cm,bottom=4cm,headheight=13.6pt]{geometry}\usepackage{amsfonts}
\usepackage{amssymb}
\usepackage{fancyhdr}
\usepackage{enumerate}
\usepackage{amsmath,amsthm}
\usepackage[dvips]{graphicx}
\usepackage{color}
\usepackage{eepic}
\usepackage{epic}
\usepackage[dvips]{rotating}
\usepackage{sectsty}
\usepackage{wrapfig}

\theoremstyle{plain}
\newtheorem{theorem}{Theorem}[section]

\theoremstyle{definition}

\newtheorem{remark}[theorem]{Remark}

\theoremstyle{remark}

 \numberwithin{equation}{section} 

\begin{document}
\title{The Blow up Rate Estimates for a Reaction Diffusion System with Gradient Terms}

\author{Maan A. Rasheed and Miroslav Chlebik}
\maketitle

\abstract 
We concider, the blow-up solutions for a coupled reaction diffusion system with gradient terms. The main purpose is to understand whether the gradient terms effect the blow-up properties. We derive the upper and lower blow-up rate estimates under certain assumptions. 

\section{\bf Introduction}
In this section, we consider the Cauchy (Dirichlet) parabolic problem: 
\begin{equation}\label{f1}\left. \begin{array}{lll}
u_t= \Delta u+|\nabla u|^{q_1} + v^{p_1},& v_t= \Delta v+|\nabla v|^{q_2} + u^{p_2}&\mbox{in}\quad \Omega\times (0,T),\\
u(x,0)=u_0(x),& v(x,0)=v_0(x),& \mbox{in}\quad \Omega,
\end{array} \right\}  \end{equation}
where $p_1,p_2,\in (1,\infty), q_1,q_2\in(1,2],$ $u_0,v_0\ge 0$ are nonzero smooth bounded functions on $\Omega$ (not necessarily radial), $\Omega=R^n$ or $B_R.$ Moreover, in case of $\Omega=B_R,$ $u,v$ are further required to satisfy the condition:
 \begin{equation}\label{ m}\begin{array}{lll}  u(x,t)=0,&\quad v(x,t)=0, &\mbox{on} \quad\partial \Omega  \times [0,T).\end{array} \end{equation}

  The problems of  semilinear parabolic equations have been studied by many authors, for instance, consider the Cauchy (Dirichlet) problem of the semilinear heat equation:
\begin{equation}\label{f3}u_t= \Delta u + u^p,  \quad \mbox{in}\quad  \Omega\times (0,T),\end{equation}
where  $p>1,$ $\Omega=R^n$ or $B_R.$ It is well known that every positive solution blows up in finite time, if the initial data is nonnegative and suitably large \cite{32,39}.
Moreover, it was proved in \cite{1,9} that the blow-up rate estimate for (\ref{f3}) takes the following form
 $$u(x,t) \le c(T-t)^{-\frac{1}{p-1}}, \quad (x,t)\in \Omega\times (0,T).$$

Later, in \cite{33} it has been shown that if we add a positive gradient term to the equation (\ref{f3}), namely 
\begin{equation}\label{f4}
u_t= \Delta u+|\nabla u|^q + u^p,
\end{equation}
 then that enhancing blow-up, and the influence of the gradient term becoming more important as the value of $p$ decreases. In the case  $q=2$ for radial positive solutions in $R^n,$ it was shown in \cite{34,35} that blow-up solutions behave asymptotically like the self-similar solution of the \emph{Hamilton-Jacobi equation} without diffusion 
($u_t=|\nabla u|^2 + u^p$),
which takes the form $$u(x,t)=(T-t)^{\frac{-1}{p-1}}w(\frac{x}{(T-t)^m}),\quad m=(2-p)/2(p-1),$$
where $w \in C^2(R^n)$ is a positive radial decreasing function.
On the other hand, the existence of nonnegative global solutions  is shown in \cite{Awos} for small initial data.

 In \cite{36,A}, it was considered, the Cauchy (Dirichlet) problem of the following semilinear system:
\begin{equation}\label{f5}
u_t= \Delta u + v^{p_1},  \quad v_t= \Delta v + u^{p_2}, \quad (x,t)\in \Omega \times (0,T),
\end{equation}
where $p_1,p_2>1,$ $\Omega=B_R$ or $R^n,$  with nonzero initial data $u_0,v_0\ge0,$ it was shown that any positive solution of this problem blows up in finite time if the initial data are large enough. Moreover, for the Cauchy problem of (\ref{f5}), it is well known \cite{36} that any nontrival positive solution blows up in finite time, if\begin{equation}\label{Zax} \max\{\alpha,\beta\}\ge \frac{n}{2},\end{equation} where
\begin{equation}\label{f6}
\alpha=\frac{p_1+1}{p_1p_2-1} ,\quad \beta=\frac{p_2+1}{p_1p_2-1}.
\end{equation}
The blow-up rate estimates of this system was studied in \cite{27,7}, it was proved that there exist a positive constant $C$ such that
$$u(x,t)\le C(T-t)^{-\alpha},\quad (x,t)\in \Omega\times (0,T),$$
$$v(x,t)\le C(T-t)^{-\beta},\quad (x,t)\in \Omega\times (0,T).$$
 
 In this paper, for problem (\ref{f1}), under some restricted assumptions,  we prove that the upper blow-up rate estimates of the positive solutions and their gradient terms, take the following forms:
\begin{align*}
u(x,t)+|\nabla u(x,t)|^{\frac{2(p_1+1)}{p_1p_2+2p_1+1}}\le C_1(T-t)^{-\alpha},\quad (x,t)\in\Omega\times(0,T),\\
 v(x,t)+|\nabla v(x,t)|^{\frac{2(p_2+1)}{p_1p_2+2p_2+1}}\le C_2(T-t)^{-\beta},\quad (x,t)\in\Omega\times(0,T),\end{align*}
 where $C_1,C_2>0.$

\section{Preliminaries}
Set $$F_1(v,\nabla u)=|\nabla u|^{q_1} + v^{p_1},\quad F_2(u,\nabla v)=|\nabla v|^{q_2} + u^{p_2}.$$
Since the system (\ref{f1}) is uniformly parabolic and $F_1,F_2$ are $C^{1}([0,\infty) \times R^n),$ moreover, the growth of the nonlinearities $F_1$ and $F_1$ with respect to the gradient is sub-quadratic, it follows that, the local existence of the unique nonnegative classical solutions to the Dirichlet (Cauchy) problem of (\ref{f1}) is guaranteed, for  smooth and bounded initial data $u_0,v_0,$ by the standard parabolic theory \cite{37} (see also \cite{JH}). On the other hand, the positive solutions of problem  (\ref{f1}) may blow up in finite time, and that due to the known blow-up results of the system (\ref{f5}) and the maximum principle \cite{2}.
\begin{remark}
Since the growth of the nonlinear terms in problem (\ref{f1}) with respect to the gradients is sub-quadratic, the gradient functions $\nabla u,\nabla v$ are bounded as long as the solution $(u,v)$ is bounded (see \cite{JH}).
\end{remark}

\subsection{Blow-up Rate Estimates}
In the next theorem, we establish the upper blow-up rate estimates for the problem (\ref{f1}). Furthermore, without comparing the blow-up solutions of this problem with those of problem (\ref{f5}), we show that the blow-up can only occur simultaneously. 
\begin{theorem}\label{f}
If $p_1,p_2,q_1$ and $q_2$ satisfy the following conditions
\begin{enumerate}[\rm(1)]
\item $\max\{\alpha, \beta\} \ge \frac{n}{2},$
\item $1< q_1< \frac{2\alpha+2}{2\alpha+1},\quad 1<q_2< \frac{2\beta+2}{2\beta+1},$\end{enumerate}
 where 
 $\alpha,\beta$ are given in (\ref{f6}), then for any positive blow-up solution $(u,v)$ of problem (\ref{f1}) there exist positive constants $C_1,C_2$ such that
\begin{equation}\label{Ram1}u(x,t)+|\nabla u(x,t)|^{\frac{2(p_1+1)}{p_1p_2+2p_1+1}}\le C_1(T-t)^{-\alpha},\end{equation}
  \begin{equation}\label{Ram2}v(x,t)+|\nabla v(x,t)|^{\frac{2(p_2+1)}{p_1p_2+2p_2+1}}\le C_2(T-t)^{-\beta},\end{equation}  in $\Omega\times(0,T)$, where $T<\infty$ is the blow-up time.  \end{theorem}

\begin{proof}
Let $$M_u(t)=\sup_{\Omega \times (0,t]}[u(x,t)+|\nabla u(x,t)| ^{\frac{2(p_1+1)}{p_1p_2+2p_1+1}}],$$ $$M_v(t)=\sup_{\Omega \times (0,t]}[v(x,t)+|\nabla v(x,t)|^{\frac{2(p_2+1)}{p_1p_2+2p_2+1}}],$$ for $t \in(0,T).$ 

Clearly, $M_u,M_v$ are positive, continuous and nondecreasing functions on $(0,T).$  At least one of them diverges as $t\rightarrow T,$ due to $(u,v)$ blows up at time $T.$

We show later that there is $\delta \in (0,1)$ such that
\begin{equation}\label{f7}
\delta\le M_u^{-\frac{1}{2\alpha}}(t)M_v^{\frac{1}{2\beta}}(t) \le \frac{1}{\delta},\quad t \in (T/2,T).
\end{equation}
So that, consequently, both $M_u,M_v$ have to diverge as $t\rightarrow T.$ 

To establish the blow-up rate estimates, we use a scaling argument similar as in \cite{27}.The proof is divided into several steps.

\hspace{-0.5 cm}{\bf Step 1: Scaling }

If $M_u$ diverges as $t\rightarrow T,$ the following procedure can be applied.

Given $t_0 \in (0,T)$, choose $(x^*,t^*)\in \Omega \times(0,t_0]$ such that
\begin{equation}\label{f8}
u(x^*,t^*)+|\nabla u(x^*,t^*)|^{\frac{2(p_1+1)}{p_1p_2+2p_1+1}} \ge \frac{1}{2} M_u(t_0).
\end{equation}
Let $\gamma=\gamma(t_0)=M_u^{-\frac{1}{2\alpha}}(t_0)$ be a scaling factor. Define the rescaled functions
\begin{equation}\label{f10}
\varphi_1^{\gamma}(y,s)=\gamma^{2\alpha}u(\gamma y+x^*,\gamma^2s+t^*),
\end{equation}
\begin{equation}\label{f11}
\varphi_2^{\gamma}(y,s)=\gamma^{2\beta}v(\gamma y+x^*,\gamma^2s+t^*),
\end{equation}
for $(y,s) \in\Omega_\gamma \times (-\gamma^{-2}t^*,\gamma^{-2}(T-t^*)),$ where $$\Omega_\gamma=\{y \in R^n :\gamma y+x^* \in \Omega\}.$$ 
Clearly, $$\Omega_\gamma := \left\{ \begin{array}{lll} R^n &\quad \mbox{if}&  \Omega=R^n,\\ B_{\frac{R}{\gamma}} &\quad \mbox{if}& \Omega=B_R. \end{array} \right.$$
Next, we aim to show that $(\varphi_1^{\gamma},\varphi_2^{\gamma})$ is a solution of the following system
\begin{equation}\label{f12} \left.\begin{array}{ll}
\varphi_{1s}^{\gamma}-\Delta\varphi_1^{\gamma}=\gamma^{\mu_1}|\nabla \varphi_1^{\gamma}|^{q_1}+(\varphi_2^{\gamma })^{p_1},\\
 \varphi_{2s}^{\gamma}-\Delta\varphi_2^{\gamma}=\gamma^{\mu_2}|\nabla\varphi_2^{\gamma}|^{q_2}+(\varphi_1^{\gamma })^{p_2},
\end{array}\right\} \end{equation}
where $\mu_1=2\alpha +2 -(2\alpha +1)q_1,~\mu_2=2\beta +2 -(2\beta +1)q_2.$ 

From the assumption (2), it follows that $\mu_1,\mu_2>0.$ 

Clearly,
\begin{equation}\label{f13}
\varphi _{1s}^{\gamma}=\gamma ^{2\alpha +2}u, \quad \nabla\varphi_1^{\gamma} = \gamma^{2\alpha +1}\nabla u, \quad \Delta \varphi_1^{\gamma} =\gamma^{2\alpha +2} \Delta u.
\end{equation}
From (\ref{f1}), (\ref{f13}), it follows
$$\frac{1}{\gamma^{(2\alpha +2)}}\varphi _{1s}^{\gamma}=\frac{1}{\gamma^{(2\alpha +2)}}\Delta \varphi_1^{\gamma}+ \frac{1}{\gamma^{q_1(2\alpha +1)}}|\nabla\varphi_1^{\gamma}|^{q_1}+\frac{1}{\gamma^{2p_1\beta} }(\varphi_2^{\gamma})^{p_1}.$$
 Multiply the last equation by ${\gamma^{(2\alpha +2)}},$ we get the first equation of (\ref{f12}). In the same way we can show that $\varphi_2^{\gamma}$ satisfies the second equation in system (\ref{f12}).
  
 Restrict $s$ to $s\in (-\gamma^{-2}t^*,0],$ our aim now is to show that
\begin{equation}\label{f15}
\varphi_1^{\gamma}(y,s)+|\nabla\varphi_1^{\gamma}(y,s)|^{\frac{2(p_1+1)}{p_1p_2+2p_1+1}}\le 1,
\end{equation}
for $(y,s) \in \Omega_\gamma \times (-\gamma^{-2}t^*,0].$

From (\ref{f13}), we obtain \begin{eqnarray}|\nabla \varphi_1^{\gamma}(y,s)|^{\frac{2(p_1+1)}{p_1p_2+2p_1+1}}&=&\gamma^{[\frac{2(p_1+1)}{p_1p_2-1}+1][\frac{2(p_1+1)}{p_1p_2+2p_1+1}]}
|\nabla u| ^{\frac{2(p_1+1)}{p_1p_2+2p_1+1}}, \nonumber\\
\label{f16} &=&\gamma ^{2\alpha}|\nabla u|^{\frac{2(p_1+1)}{p_1p_2+2p_1+1}}.\end{eqnarray}
Clearly, \begin{equation}\label{f114}
 u(x,t)+|\nabla u(x,t)|^{\frac{2(p_1+1)}{p_1p_2+2p_1+1}}\le M_u(t_0), \quad (x,t) \in \Omega\times(0,t^*].
  \end{equation}
From (\ref{f10}), (\ref{f16})  and (\ref{f114}),  we get (\ref{f15}).
 
 Moreover, \begin{equation}\label{} \varphi_2^{\gamma}+|\nabla\varphi_2^{\gamma}|^{\frac{2(p_2+1)}{p_1p_2+2p_2+1}}\le M_u^{-\frac{\beta}{\alpha}}(t_0)M_v(t_0),\end{equation}  for $(y,s) \in \Omega_\gamma \times (-\gamma^{-2}t^*,0].$
 
  On the other hand, from (\ref{f8}), we obtain 
\begin{equation}\label{f17}
\varphi_1^{\gamma}(0,0)+|\nabla \varphi_1^{\gamma}(0,0)|^{\frac{2(p_1+1)}{p_1p_2+2p_1+1}}\ge \frac{1}{2}. \end{equation}
If $M_v$ diverges as $t\rightarrow T$ we can proceed in the same way by changing the role of $u$ and $v.$ 

\hspace{-0.5 cm}{\bf Step 2: Schauder's estimates} 

We need interior Schauder's estimates of the functions $\varphi_1,\varphi_2$ on the sets
$$S_{K}=\{y \in \Omega_\gamma ,|y| \le K \} \times [-K,KL],\quad K>0,~L=0,1.$$ 
    Assume that $\varphi_1,\varphi_2$ satisfy in $S_{2K}$ the condition \begin{equation}\label{af}
0 \le \varphi_1^{\gamma}+|\nabla \varphi_1^{\gamma}|^{\frac{2(p_1+1)}{p_1p_2+2p_1+1}}\le B,\quad 0 \le \varphi_2^{\gamma}+|\nabla \varphi_2^{\gamma}|^{\frac{2(p_2+1)}{p_1p_2+2p_2+1}}\le B. \end{equation}
We claim that  for any $ K> 0,B>0$ and $\sigma >0$ small enough, there is a constant $C=C(K,B,\sigma)$ such that
\begin{equation}\label{f19}
||\varphi_1^{\gamma}||_{C^{2+\sigma,1+\frac{\sigma}{2}}(S_K)} \le C, \quad
||\varphi_2 ^{\gamma}||_{C^{2+\sigma,1+\frac{\sigma}{2}}(S_K)} \le C.
\end{equation}
From (\ref{af}) we deduce that each of $\varphi_1^\gamma,\varphi_2^\gamma,\nabla\varphi_1^\gamma,\nabla\varphi_2^\gamma$, is uniformly bounded function in $S_{2K}.$ Therefore,  the functions $(\varphi_1^\gamma)^{p_1},(\varphi_2^\gamma)^{p_2},|\nabla\varphi_1^\gamma|^{q_1},|\nabla\varphi_2^\gamma|^{q_2}$ are uniformly bounded in $S_{2K}.$ So, the right hand sides of the two equations in (\ref{f12}) are uniformly bounded functions in $S_{2K}$, applying the interior reqularity theory (see \cite{37}), we obtain (locally) uniform estimates in $C^{1+\sigma ,\frac{1+\sigma}{2}}$-norms. Consequently, by Lemma \ref{Bz}, we obtian (locally) uniform estimates in H\"{o}lder norms $C^{\sigma,\frac{\sigma}{2}}$ on the right hand sides of the both equations in (\ref{f12}).Thus the parabolic interior Schauder's estimates imply (\ref{f19}) (see \cite{23,37}).

 \hspace{-0.5 cm}{\bf Step 3: The proof of (\ref{f7})}
 
Suppose that this lower bound were false.Then there exist a sequence $t_j \rightarrow T$ such that
\begin{equation}\label{f20}
M_u^{-\frac{1}{2\alpha}}(t_j)M_v^{\frac{1}{2\beta}}(t_j) \longrightarrow 0,\quad \mbox{as} ~j\rightarrow \infty.
\end{equation}
Then clearly $M_u$ diverges as $t_j \rightarrow T$.
For each $t_j$  in the role of $t_0$ from {Step 1}, we scale about the correspoinding point $(x_j^*,t_j^*)$ for all $j$ such that $t_j^*\le t_j,$ with the scaling factor
$$\gamma_j=\gamma(t_j)=M_u^{-\frac{1}{2\alpha}}(t_j).$$
We obtain the corresponding rescaled solution $(\varphi_1^{\gamma_j},\varphi_2^{\gamma_j}),$
\begin{equation}
\varphi_1^{\gamma_j}(y,s)=\gamma_j^{2\alpha}u(\gamma_j y+x_j^*,\gamma_j^2s+t_j^*),
\end{equation}
\begin{equation}
\varphi_2^{\gamma_j}(y,s)=\gamma_j^{2\beta}v(\gamma_j y+x_j^*,\gamma_j^2s+t_j^*).
\end{equation}
Clearly, $(\varphi_1^{\gamma_j},\varphi_2^{\gamma_j})$ satisfies (as in {Step 1}) the following problem \begin{equation}\label{f90}\left.\begin{array}{ll}  \varphi_{1s}^{\gamma_j}-\Delta\varphi_1^{\gamma_j}=\gamma_j^{\mu_1}|\nabla \varphi_1^{\gamma_j}|^{q_1}+(\varphi_2^{\gamma_j })^{p_1},\\
 \varphi_{2s}^{\gamma_j}-\Delta\varphi_2^{\gamma_j}=\gamma_j^{\mu_2}|\nabla \varphi_2^{\gamma_j}|^{q_2}+(\varphi_1^{\gamma_j })^{p_2},\end{array}\right\}\end{equation}  with 
\begin{equation}\label{f81}\left. \begin{array}{ll} \varphi_1^{\gamma_j}(0,0)+|\nabla \varphi_1^{\gamma_j}(0,0)|^{\frac{2(p_1+1)}{p_1p_2+2p_1+1}}\ge 1/2,\\
0 \le \varphi_1^{\gamma_j}+|\nabla \varphi_1^{\gamma_j}|^{\frac{2(p_1+1)}{p_1p_2+2p_1+1}}\le 1, \\
  \varphi_2^{\gamma_j}+|\nabla \varphi_2^{\gamma_j}|^{\frac{2(p_2+1)}{p_1p_2+2p_2+1}}\ \le M_u^{-\frac{\beta}{\alpha}}(t_j)M_v(t_j), \end{array}\right\}\end{equation}
for $(y,s) \in \Omega_{\gamma_j} \times ( -\gamma_j^{-2}t_j^*,0],$ where 
$$\Omega_{\gamma_j} := \left\{ \begin{array}{lll} R^n &\quad \mbox{if}&  \Omega=R^n,\\ B_{\frac{R}{\gamma_j}} &\quad \mbox{if}& \Omega=B_R. \end{array} \right.$$
Clearly, $$\Omega_{\gamma_j}\longrightarrow R^n,\quad\mbox{as}~ j\rightarrow \infty.$$
From (\ref{f20}), (\ref{f81}), we see that $$ \varphi_2^{\gamma_j}+|\nabla \varphi_2^{\gamma_j}|^{\frac{2(p_2+1)}{p_1p_2+2p_2+1}} \longrightarrow 0,\quad \mbox{as} ~j\rightarrow \infty.$$ Thus $\varphi_2^{\gamma_j},\nabla \varphi_2^{\gamma_j}$ are bounded in $\Omega_{\gamma_j} \times ( -\gamma_j^{-2}t_j^*,0], ~\forall j.$ 

Using the uniform Schauder's estimate derived in {Step 2} to $(\varphi_1^{\gamma_j},\varphi_2^{\gamma_j})$
$$||\varphi_1^{\gamma_j}||_{C^{2+\sigma,1+\frac{\sigma}{2}}( \{y \in \Omega_{\gamma_j} ,|y|\le K \} \times [ -K,0])}\le C_K,$$
$$|| \varphi_2 ^{\gamma_j}||_{C^{2+\sigma,1+\frac{\sigma}{2}}( \{y \in \Omega_{\gamma_j},|y|\le K\} \times [ -K,0])}\le C_K,$$ where $C_K$ is independent of $j .$  

Since $(\varphi_1^{\gamma_j},\varphi_2^{\gamma_j})$ is defined on a compact set, by the Arzela-Ascoli theorem, there exist a convergent subsequance, we still denote it by  $(\varphi_1^{\gamma_j},\varphi_2^{\gamma_j}).$ 

Since $\mu_1,\mu_2 >0$ and $\nabla \varphi_1^{\gamma_j},\nabla \varphi_2^{\gamma_j}$ are bounded, it follows that, the limit point  $(\varphi _1,\varphi_2)$ is a solution  of the following system
 \begin{equation}\label{f22}
\varphi_{1s}=\Delta\varphi_1+\varphi_2^{p_1},\quad
 \varphi_{2s}=\Delta\varphi_2+\varphi_1^{p_2},\quad \mbox{in}\quad R^n \times (-\infty ,0].
 \end{equation}

Since $\varphi_2^{\gamma_j} \rightarrow 0,~ \mbox{as}~ j \rightarrow \infty,$ it follows that $\varphi_2\equiv 0,~\mbox{in}~R^n \times (-\infty ,0].$ 

Consequently, from the second equation in (\ref{f22}), we obtain that  $$\varphi_1\equiv 0,\quad \mbox{in}\quad R^n \times (-\infty ,0].$$

This means $$\varphi_1(0,0)+|\nabla\varphi_1(0,0)|^{\frac{2(p_1+1)}{p_1p_2+2p_1+1}}=0,$$ 
which contradicts with (\ref{f81}). Thus, the lower bound is held.

To prove the upper bound of (\ref{f7}) we proceed similarly as in the proof of lower bound with changing the role of $u$ and $v$. 

\hspace{-0.5 cm}{\bf Step 4: Estimate on doubling of  $M_u$}

 As $M_u$ is continuous and diverges as $t \rightarrow T,$ for any $t_0 \in (0,T)$ we define $t_0^+$ by
$$t_0^+=\max\{ t \in (t_0,T) : M_u(t)=2M_u(t_0)\}.$$ 
Clearly,
\begin{equation}\label{f14}
 u(x,t)+|\nabla u(x,t)|^{\frac{2(p_1+1)}{p_1p_2+2p_1+1}}\le 2M_u(t_0), \quad (x,t) \in \Omega\times(0,t_0^+].
  \end{equation}
  Take $\gamma =\gamma(t_0)=M_u^{-\frac{1}{2\alpha}}(t_0).$ 

We claim that
$$\gamma^{-2}(t_0)(t_0^+-t_0) \le A, \quad t_0 \in (\frac{T}{2},T),$$
where the constant $A \in (0,\infty)$ is independent of $t_0.$
Suppose that this estimate were false, then there would exist a sequence $t_j \rightarrow T$ such that
$$\gamma_j^{-2}(t_j)(t_j^+-t_j) \rightarrow \infty,$$
where \begin{equation}\label{raf} t_j^+=\max\{ t \in (t_j,T): M_u(t)=2M_u(t_j)\}.\end{equation}
For each $t_j$ we scale about the corresponding point $(x_j^*,t_j^*)$ such that $$0<t_j^*\le t_j,\quad\frac{T}{2}< t_j< t_j^+<T,\quad \forall j$$  with the scaling factor $$\gamma_j =\gamma (t_j)=M_u^{-\frac{1}{2\alpha}}(t_j).$$
As in {Step 3}, we obtain the corresponding rescaled functions  $(\varphi_1^{\gamma_j},\varphi_2^{\gamma_j}),$ which satisfies  (\ref{f90}) with the following conditions

 \begin{equation}\label{xza}\left. \begin{array}{ll} \varphi_1^{\gamma_j}(0,0)+|\nabla \varphi_1^{\gamma_j}(0,0)|^{\frac{2(p_1+1)}{p_1p_2+2p_1+1}}\ge 1/2,\\
0 \le \varphi_1^{\gamma_j}+|\nabla \varphi_1^{\gamma_j}|^{\frac{2(p_1+1)}{p_1p_2+2p_1+1}}\le 2, \\
  \varphi_2^{\gamma_j}+|\nabla \varphi_2^{\gamma_j}|^{\frac{2(p_2+1)}{p_1p_2+2p_2+1}}\ \le M_u^{-\frac{\beta}{\alpha}}(t_j)M_v(t_j^+), \end{array}\right\}\end{equation}
   for $(y,s)\in\Omega_{\gamma_j}\times(-\gamma_j^{-2}t^*,\gamma_j^{-2}(t_j^+-t_j^*)].$ 
 
 From (\ref{raf}) and (\ref{xza}), it follows that \begin{equation}\label{f26} \varphi_2^{\gamma_j}+|\nabla\varphi_2^{\gamma_j}|^{\frac{2(p_2+1)}{p_1p_2+2p_2+1}}\ \le 2^{\frac{\beta}{\alpha}}M_u^{-\frac{\beta}{\alpha}}(t_j^+)M_v(t_j^+).\end{equation}  From (\ref{f7}), we have
$$M_v(t) \le  \delta ^{-2\beta} M_u^{\frac{\beta}{\alpha}}(t), \quad t\in (\frac{T}{2},T).$$
Therefore, (\ref{f26}) becomes
$$\varphi_2^{\gamma_j}+|\nabla\varphi_2^{\gamma_j}|^{\frac{2(p_2+1)}{p_1p_2+2p_2+2}}\ \le \frac{2^{\frac{\beta}{ \alpha}}}{\delta^{2\beta}}.$$
By using the Schauder's estimates derived in {Step 2} for $(\varphi_1^{\gamma_j},\varphi_2^{\gamma_j})$ we get a convergent subseguence in $C_{loc}^{2+\sigma,1+\sigma /2}(R^n \times R)$ to the solution of system (\ref{f22}) in $R^n \times R.$ This is a contradiction because all the nontrival positive solutions of system (\ref{f22}), under the assumption (1), blow up in finite time (see \cite{36}).

Thus, there is $A>0$ such that \begin{equation}\label{f23} \gamma^{-2}(t_0)(t_0^+-t_0) \le A, \quad t_0\in(\frac{T}{2},T).\end{equation}
\hspace{0 cm}{\bf Step 5: Rate estimates }

As in {Step 4}, for any $t_0 \in(T/2,T)$ we define
$$t_1=t_0^+ \in(t_0,T)\quad  \mbox{such that}\quad  M_u(t_1)=2M_u(t_0).$$
Due to (\ref{f23}), $$(t_1-t_0)\le A M_u^{-\frac{1}{\alpha}}(t_0).$$ We can use $t_1$ as a new $t_0$ and obtain $t_2 \in (t,T)$ such that $$M_u(t_2)=2M_u(t_1)=4M_u(t_0),$$
$$(t_2-t_1)\le A M_u^{-\frac{1}{\alpha}}(t_1)=2^{-\frac{1}{\alpha}}A M_u^{-\frac{1}{\alpha}}(t_0).$$
Continuing this process we obtain a sequence $t_j \rightarrow T$ such that
$$(t_{j+1}-t_j)\le 2^{-\frac{j}{\alpha}}A M_u^{-\frac{1}{\alpha}}(t_0),\quad j=0,1,2,\ldots$$
If we add these inequalities we get
$$(T-t_0)\le \sum_{j\ge 0} 2^{-\frac{j}{\alpha}}AM_u^{-\frac{1}{\alpha}}(t_0).$$
Thus $$(T-t_0)\le (1-2^{-\frac{1}{\alpha}})^{-1}AM_u^{-\frac{1}{\alpha}}(t_0)$$
From using (\ref{f7}) we obtain
$$M_v(t_0)\le \delta^{-2\beta}M_u^{\frac{\beta}{\alpha}}(t_0), ~t_0 \in (T/2,T).$$
Thus $$M_v(t_0)\le \delta^{-2\beta}(1-2^{-\frac{1}{\alpha}})^{-\beta}A^\beta (T-t_0)^{-\beta},\quad t_0 \in (T/2,T).$$
From above there exist two constants $C_1^*,C_2^*$ such that
$$M_u(t_0)\le C_1^*(T-t_0)^{-\alpha},\quad t_0\in (\frac{T}{2},T),$$
$$M_v(t_0)\le C_2^*(T-t_0)^{-\beta},\quad t_0\in (\frac{T}{2},T).$$
From the last two equations and the definitions of $M_u,M_v,$ it follows that there exist constants $C_1,C_2$ such that
$$u(x,t)+|\nabla u(x,t)| ^{\frac{2(p_1+1)}{p_1p_2+2p_1+1}}\le C_1(T-t)^{-\alpha},$$
 $$v(x,t)+|\nabla v(x,t)|^{\frac{2(p_2+1)}{p_1p_2+2p_2+1}}\le C_2(T-t)^{-\beta},$$
for $(x,t) \in \Omega \times (0,T).$
\end{proof}
\begin{remark}
If $u_0\equiv v_0,$ $p=p_1=p_2,$ $q=q_1=q_2,$ then problem (\ref{f1}) can be reduced to a scalar Dirichlet (Cauchy) problem for (\ref{f4}). Moreover, if 
\begin{equation}\label{Abbas} 1< p\le 1+\frac{2}{n},\quad  1<q<\frac{2p}{1+p},\end{equation} 
then in a similar way to the proof of Theorem \ref{f}, we can show that, for a nontrivial positive blow-up solution $u,$ there exist $C>0$ such that 
\begin{equation}\label{Fri} u(x,t)+|\nabla u(x,t)|^{\frac{2}{p+1}}\le C(T-t)^{\frac{1}{p-1}},\quad \mbox{in}\quad \Omega\times(0,T),\end{equation}
i.e.  \begin{equation}\label{hoda}u(x,t)\le C(T-t)^{\frac{1}{p-1}},\quad \mbox{in}\quad \Omega\times(0,T).\end{equation}
As we have mentioned before, the rate estimate (\ref{hoda}) is also known for the blow-up solutions of equations  (\ref{f3}). Therefore, if $p,q$ satisfy (\ref{Abbas}), then the positive gradient terms which appear in equation (\ref{f4}), does not affect the blow-up rate estimates of these problems.  A similar observation holds for  problem (\ref{f1}) by 
Theorem \ref{f},  which shows that the upper rate estimates of the Cauchy or Dirichlet problem for system (\ref{f1}) are the same as those known for the system (\ref{f5}). Therefore, under the assumptions of Theorem \ref{f}, the gradient terms in system (\ref{f1}) have no effect on the blow-up rate estimates. 
\end{remark}
\subsection{Blow-up Set}
It is well known that for the semilinear system (\ref{f5}) defined in a ball, under some restricted assumptions on $u_0,v_0$ (nonnegative, radially decreasing functions), that the only blow-up point is the centre of that ball (see \cite{10}), while it is unknown whether this holds for the system (\ref{f1}). However, for the radial solutions of the single equation (\ref{f4}) defined in $\Omega,$ in case $q=2,$ there is global blow-up, if $1<p<2,$ $\Omega=B_R$ or $R^n,$ and regional blow-up, if $p=2,$ $\Omega=R^n,$ while a single blow-up point, if $p>2,$ $\Omega=B_R$ (see \cite{2,26} and the references therein).The proof relies on the transformation $v=e^u-1,$ which converts (\ref{f4}) into the semilinear heat equation $v_t=\Delta v+(1+v)\log^p(1+v).$
We note that, these results are much different from those known for equation (\ref{f3}) (see \cite{2}), because for any $p>1,$ $\Omega=B_R$ or $R^n,$ only a single blow-up point is known to occur for that problem, where the initial date are nonnegative, radially nonincreasing and bounded function.

\end{document}